\documentclass[12pt]{article}
\usepackage[dvips]{graphics}   
\usepackage{times}     
\usepackage[T1]{fontenc}  
\usepackage{graphicx,enumerate,amsmath,amsthm,mathrsfs,array,pstricks}
\usepackage{amsfonts}
\usepackage{amssymb}
\usepackage[a4paper,left=2.5cm,right=3cm,top=2.5cm,bottom=2.5cm]{geometry}

\theoremstyle{}
\newtheorem{thm}{Theorem}[section]
\newtheorem{cor}[thm]{Corollary}

\newtheorem{prop}[thm]{Proposition}
\newtheorem{prob}[thm]{Problem}
\newtheorem{conj}[thm]{Conjecture}

\theoremstyle{definition}

\theoremstyle{remark}
\newtheorem{rem}{Remark}[section]
\def\cart{\mathbin\square}

\begin{document}
\begin{titlepage}
\title{On Domatic and Total Domatic Numbers of Product Graphs}
\author{P. Francis$^{1}$ and Deepak Rajendraprasad$^{2}$}
\date{{\footnotesize Department of Computer Science, Indian Institute of Technology Palakkad, India.}\\
{\footnotesize $^{1}$: pfrancis@iitpkd.ac.in\  $^{2}$: deepak@iitpkd.ac.in }}
\maketitle
\renewcommand{\baselinestretch}{1.2}\normalsize

\begin{abstract}
A \emph{domatic} (\emph{total domatic}) \emph{$k$-coloring} of a graph $G$ is an assignment of $k$ colors
 to the vertices of $G$ such that each vertex contains vertices of all $k$ colors in its closed neighborhood (neighborhood). The \emph{domatic} (\emph{total domatic}) \emph{number} of $G$, denoted $d(G)$ ($d_t (G)$), is the maximum $k$ for which $G$ has a domatic (total domatic) $k$-coloring.
In this paper, we show that for two non-trivial graphs $G$ and $H$, the domatic and
total domatic numbers of their Cartesian product $G \cart H$ is bounded
above by $\max\{|V(G)|, |V(H)|\}$ and below by $\max\{d(G), d(H)\}$. Both these
bounds are tight for an infinite family of graphs. Further, we show that if $H$ is bipartite, then $d_t(G \cart H)$ is
bounded below by $2\min\{d_t(G),d_t(H)\}$ and $d(G \cart H)$ is
bounded below by $2\min\{d(G),d_t(H)\}$. These bounds give easy proofs for many of the known
bounds on the domatic and total domatic numbers of hypercubes \cite{chen,zel4} and the domination and total domination numbers of hypercubes \cite{har,joh} and also give new
bounds for Hamming graphs.
We also obtain the domatic (total domatic) number and domination (total domination) number of $n$-dimensional torus $\mathop{\cart}\limits_{i=1}^{n} C_{k_i}$ with some suitable conditions to each $k_i$, which turns out to be a generalization of a result due to  Gravier \cite{grav2}
and give easy proof of a result due to Klav\v{z}ar and Seifter \cite{sand}.
\end{abstract}

\section{Introduction}

All graphs considered in this paper are finite, simple, undirected and do not contain an isolated vertex.
Let $P_n,C_n$ and $ K_n$ respectively denote the path, the cycle and the complete graph on $n$ vertices. Let $\delta(G)$ and $\Delta(G)$ denote the minimum and maximum degree of a graph $G$ respectively.
The neighborhood $N(x)$ of a vertex $x$ is $\{u\colon ux\in E(G)\}$ and its closed neighborhood $N[x]$ is $N(x)\cup \{x\}$. For $S\subseteq V(G)$, 
let $\langle S\rangle$ denote the subgraph induced by $S$ in $G$. For any graph $G$, let $\overline{G}$ denote the complement of $G$.
Let $[n]$ be the set of consecutive integers $\{1,2,\ldots,n\}$ and $\mathbb{Z}_n$ be the set of congruence classes of integers modulo $n$.

The \emph{Cartesian product} of two graphs $G$ and $H$, denoted $G\cart H$, is a graph whose vertex set is $ V(G)\times V(H) =\{(x, y)\colon x\in V(G)~\mathrm{and}~ y\in V(H)\}$ and two vertices $(x_1, y_1)$ and $(x_2,y_2)$ of $G\cart H$ are adjacent if and only if either $x_1 = x_2$ and $y_1y_2 \in E(H)$ or $y_1 = y_2$ and $x_1x_2\in E(G)$.
For any vertex $u\in V(G)$, $\langle  \{u\}\times V(H)\rangle$  is isomorphic to $H$. It is called the $H$-layer of $u$ and is  denoted by $H_u$. For any vertex $v\in V(H)$, $\langle V(G)\times \{v\}\rangle$ is isomorphic to $G$ called the $G$-layer of $v$ and is denoted by $G_v$. For $d\geq2$, let $\mathop{\cart}\limits_{i=1}^{d}G_i$ denotes $G_1\cart G_2\cart\cdots \cart G_d$.
For $r_i\geq3$,  we call $\mathop{\cart}\limits_{i=1}^{d}C_{r_i}$ a $d$-dimensional \emph{torus}.
For $n\geq1,q\geq2$, the \emph{Hamming graph} $H_{n,q}$ is $\mathop{\cart}\limits_{i=1}^{n}K_{q}$.
The special case $H_{n,2}$ is a hypercube of dimension $n$, denoted as $Q_n$.

A \emph{domatic} (\emph{total domatic}) \emph{$k$-coloring} of a graph $G$ is an assignment of $k$ colors
 to the vertices of $G$ such that each vertex contains vertices of all $k$ colors in its closed neighborhood (neighborhood). The \emph{domatic} (\emph{total domatic}) \emph{number} of $G$, denoted $d(G)$ ($d_t (G)$), is the maximum $k$ for which $G$ has a domatic (total domatic) $k$-coloring.

Let $D\subset V(G)$, if $N(D)\supseteq V(G)\backslash D$ then $D$ is a dominating set of $G$ and if $N(D)=V(G)$ then $D$ is a total dominating set of $G$. The \emph{domination} (\emph{total domination}) number of a graph $G$ is the cardinality of a smallest dominating (total dominating) set of $G$ and is denoted $\gamma(G)$ ($ \gamma_t(G)$). In any domatic (total domatic) coloring of a graph $G$, each color class is a dominating (total dominating) set of $G$.  Thus the domatic and total domatic numbers can be also seen in the following way.
The \emph{domatic} (\emph{total domatic}) number of $G$ is the maximum number of classes of a partition of $V(G)$ such that each class is a dominating (total dominating) set of $G$.
There is considerable literature on  domination and total domination in graphs. See for instance, \cite{bre,grav2,hen2,ho,sand} and a survey of selected topics by Henning \cite{hen1}.

The concept of domatic number and total domatic number was introduced by Cockayne et al., in \cite{cok} and \cite{cok1} respectively, and investigated further in \cite{akb,ara,bou,chen,god,heg,koi,nag,zel3,zel2}. In \cite{zel2}, Zelinka have shown the existence of graphs with very large minimum degree have a total domatic number $1$.
Chen et al., \cite{chen} and Goddard and Henning \cite{god} have studied total domatic coloring under the names \emph{coupon coloring} and \emph{thoroughly dispersed coloring} respectively.
The motivation for study of total domatic coloring and its applications were mentioned by Chen et al., in \cite{chen}.
Further, they showed that every $d$-regular graph $G$ has $d_t(G)\geq (1 -  o(1))d/ \log d$ as $d\rightarrow\infty$, and the proportion of
$d$-regular graphs $G$ for which $d_t(G) \leq (1 + o(1))d/ \log d$ tends to $1$ as $|V(G)|\rightarrow\infty$.  In \cite{god}, Goddard and Henning have shown that the total domatic number of a planar
graph cannot exceed $4$ and conjectured that every planar triangulation $G$ on four or more vertices has $d_t(G)$ at least $2$.
There are some partial answers to this conjecture by Akbari et al., \cite{akb} and Nagy \cite{nag}.
For a bipartite graph $G$, Heggernes and Telle \cite{heg} shown that deciding whether $d_t(G)\geq2$ is NP-complete. In \cite{koi}, Koivisto et al., shown that it is NP-complete to decide whether $d_t(G)\geq3$ where $G$ is a bipartite planar graph of bounded maximum degree. Also, they have shown that if $G$ is split or $k$-regular graph for $k\geq3$, then it is NP-complete to decide whether $d_t(G) \geq k$.

In \cite{chen}, Chen et al., mentioned that  for any graph $G$, it would be interesting to determine any relations between $d_t (G)$ and $d_t (G \cart G)$. More generally, for any graphs $G$ and $H$, we start to determine the relationship between $d_t (G)$, $d_t(H)$ and $d_t (G \cart H)$.
In this direction, we prove that if at least one among $G$ or $H$ is bipartite, then $G\cart H$ has a total domatic coloring with $2\min\{d_t(G),d_t(H)\}$ colors. As consequences, we show that when $n$ is a power of $2$, the total domatic number of the hypercubes $Q_n$ and $Q_{n+1}$ is $n$ and the torus $\mathop{\cart}\limits_{i=1}^n C_{4r_i}$ is $2n$.
Also, for any positive integer $d$ and with some suitable conditions to each $k_i$, we show that the total domatic number and total domination number of the torus $\mathop{\cart}\limits_{i=1}^{d} C_{k_i}$ is  $2d$ and $(\mathop{\prod}\limits_{i=1}^{d} {k_i})/2d$ respectively.
In addition, we obtain similar bounds and results for the domatic number and domination number of $G\cart H$.
Also, we prove that the domatic and total domatic numbers of $G\cart H$ is  upper bounded by  $\max\{|V(G)|, |V(H)|\}$, and lower bounded by   $\max\{d(G),d(H)\}$.

The concept of injective coloring was introduced by Hahn et al., in \cite{hah} and further studied in \cite{bu,cran,luz}. An \emph{injective $k$-coloring} of a graph $G$ is an assignment of $k$ colors to the vertices of $G$ such that any two vertices in the neighborhood of each vertex  have distinct colors.
The minimum $k$ for which such a coloring exists is the \emph{injective chromatic number} of $G$, denoted $\chi_i(G)$.
Also, the injective chromatic number of a graph $G$ can be seen in the following way.
The \emph{common neighbor graph} $G^{(2)}$ of $G$ has the same vertex set  $V(G)$
and any two vertices $u,v$ are adjacent in $G^{(2)}$ if there is a path of length $2$ joining $u$ and $v$ in $G$.
It is noted that $\chi_i (G)=\chi(G^{(2)})$.
The square of a graph $G$, denoted $G^2$, has the same vertex set  $V(G)$
and edge set $E(G)\cup E(G^{(2)})$.

We obtain a lower bound for the domatic and total domatic number of $H_{n-1,q}$ and $H_{n,q}$ respectively, when $n$ is a power of $2$ and $q$ at least $2$. In \cite{chen}, Chen et al., determined the injective chromatic number of $H_{n,q}$, where $q$ is a prime power and $n=\frac{q^k-1}{q-1}$, for some positive integer $k$. As a consequence, we obtain a lower bound for the domatic and total domatic numbers of $H_{n-1,q}$ and $H_{n,q}$ respectively for some more values of $n$ when $q$ is a prime power.

\section{Preliminaries}

It is easy to see that $d(G) \leq \delta(G)+1$ and $d_t(G) \leq \delta(G)$ for every graph $G$.
We will call the graphs which attain these bounds as \emph{domatically full} and \emph{total domatically full} respectively.
Regular total domatically full graphs are also called \emph{rainbow graphs} (see, \cite{oh,wol}). We first make some easy observations on rainbow
graphs. Examples of rainbow graphs include cycles
$C_n$ where $n\equiv 0\pmod 4$, $K_{n,n}$, $K_n \cart K_2$, etc.

\begin{prop}\label{DiviReg}
Let $G$ be an $r$-regular total domatically full graph.
Every $r$-total domatic coloring of $G$, say  $f \colon V(G)\rightarrow [r] $ satisfies the following.
\begin{enumerate}[(i)]
\setlength\itemsep{-2pt}
\item Each color class of $f$ has the same size $\frac{|V(G)|}{r}$. (\cite{wol,zel3})
\item  Each color class of $f$ has an even
 number of vertices and $G$ contains a perfect matching.
\item  $r$ divides $|V(G)|$ and $r^2$ divides $|E(G)|$. (\cite{wol})
\item $\gamma_t(G)=\frac{|V(G)|}{r}$.
\item $d_t(G)=\chi_i(G)=r$.
\end{enumerate}
\end{prop}

\begin{proof}
(i) Let $V_1$ and $V_2$ be two color classes of $c_1$ and $c_2$ respectively, where $c_1,c_2\in [r]$. Each vertex of $V_1$ is adjacent to exactly one vertex of $V_2$ and vice versa
and thus $|V_1|=|V_2|$. Hence $f$ partitioned $V(G)$ into $r$ classes having the same size  $\frac{|V(G)|}{r}$ (see, \cite{wol,zel3}).\\
(ii) For any vertex $x\in V(G)$, there exists exactly one neighbor of $x$ having the color $f(x)$. Thus each color class $V_i$ of $f$ induces a perfect matching in $\langle V_i\rangle$ and hence $|V_i|$ is even.  Also, $G$ contains a perfect matching which is the union of perfect matchings of all color classes.\\
(iii) Clearly, $r$ divides $|V(G)|$ and $r^2$ divides $|E(G)|$ which follows from the fact that $|E(G)|=\frac{|V(G)| r}{2}$ and $\frac{|V(G)|}{r}$ is even (see, \cite{wol}). \\
(iv) Each color class $V_i$ of $f$ is a total dominating set, thus $\gamma_t(G)\leq\frac{|V(G)|}{r}$. Also, $\gamma_t(G)\geq\frac{|V(G)|}{\Delta(G)}=\frac{|V(G)|}{r}$, since any set $S$ of size smaller than $\frac{|V(G)|}{\Delta(G)}$ can dominate at most $|S|\Delta(G) <  |V(G)|$ vertices.\\
(v) Clearly, $f$ is also an injective coloring and the proof follows from $\chi_i(G) \geq \Delta(G) = r$.
\end{proof}

A complete characterization of regular domatically full
graphs was done by Zelinka \cite{zel1}. Examples of regular domatically full graphs include cycles
$C_n$ where $n\equiv 0\pmod 3$, $K_n$, etc.

\begin{thm}\cite{zel1}\label{DiviRegDom}
An $r$-regular graph $G$ is domatically full if and only if $r+1$ divides $|V(G)|$ and $G$ has an $(r+1)$-coloring $f $ such that  each color class of $f$ is an independent set of size $\frac{|V(G)|}{r+1}$ and the subgraph induced by the vertices of any two color classes of $f$ is a perfect matching.
\end{thm}

The domatic coloring of an $r$-regular graph is closely associated with the proper injective coloring which follows in Corollary \ref{domcor}.
We say that an injective coloring is proper if no two adjacent vertices get the same color.
\begin{cor}\label{domcor}
Let $G$ be an $r$-regular graph. $G$ is domatically full if and only if $G$ has a proper $(r+1)$-injective coloring. Also, $\gamma(G)=\frac{|V(G)|}{r+1}$.
\end{cor}

\section{Domatic and total domatic numbers of Cartesian products}

The union of two disjoint dominating sets of $G$ should be a total dominating set for $G$ and thus $\left\lfloor\frac{d(G)}{2}\right\rfloor\leq d_t(G)\leq d(G)\leq {2d_t(G)+1}$. In any total domatic coloring of $G$, each vertex should receive its own color from a neighbor, thus $d_t(G)\leq \left\lfloor\frac{|V(G)}{2}\right\rfloor$ and also, $d(G)\leq |V(G)|$.
Hence it follows that, for any two graphs $G$ and $H$, $\left\lfloor\frac{1}{2}\max\{d(G),d(H)\}\right\rfloor\leq \left\lfloor\frac{1}{2}d(G\cart H)\right\rfloor\leq d_t(G\cart H)\leq \left\lfloor\frac{1}{2}|V(G\cart H)|\right\rfloor=  \left\lfloor\frac{1}{2}|V(G)||V(H)|\right\rfloor$ and $d(G\cart H)\leq |V(G\cart H)|=|V(G)||V(H)|$.
We improve these bounds given above for any two graphs  in Theorem \ref{cccartub}.

\begin{thm}\label{cccartub}
For any two graphs $G$ and $H$ without an isolated vertex, we have
\begin{center}
$\max\{d(G),d(H)\}\leq d_t(G\cart H)\leq d(G\cart H)\leq \max\{|V(G)|,|V(H)|\}$.
\end{center}
\end{thm}

\begin{proof}
Let $G$ and $H$ be graphs of order $m$ and $n$ respectively. Now, let us prove the upper bound for $d(G\cart H)$. Without loss of generality, let $n\geq m$. Let us consider the coloring of $G\cart H$ as filling the cells of $m\times n$ grid with colors. For a cell $(i,j)$, $1\leq i\leq m$, $1\leq j\leq n$, call the set of cells in the $i^{th}$ row and $j^{th}$ column as a cross-hair at $(i,j)$.
There are $mn$ cross-hairs, one corresponding to each cell of the grid. Each cross-hair has $m+n-1$ cells. If there is a $k$-domatic coloring, then each cross-hair contains all $k$ colors occurs at least once.

\noindent
\emph{Claim.} In any domatic coloring of $G\cart H$, each color should appears in at least $m$ cells.

Suppose a color $c_1$ appears in less than $m$ cells, then there exists a row $i$ as well as a column $j$ in the grid which do not contain $c_1$. In this case, the cross-hair at $(i,j)$ does not contain $c_1$ and hence the coloring is not a domatic coloring. Thus the claim holds.
\noindent
Since each color should appears in at least $m$ cells and there are $mn$ cells in the grid, the maximum possible value of $k$ in any domatic $k$-coloring is $n$. Thus $d(G\cart H)\leq n=\max\{|V(G)|, |V(H)|\}$.

Now, let us prove the lower bound for $d_t(G\cart H)$. Let  $r$ and $s$ be the domatic numbers of   $G$ and $H$ respectively. Suppose $r\geq s$. Let $D_1,D_2,\ldots, D_r$ be a domatic partition of $V(G)$. For $1\leq i\leq r$, any vertex $u\in D_i$ and $v\in V(H)$, let us define a coloring $f$ for the vertices of $G\cart H$ by  $f((u,v))=i$. Since $H$ does not have an isolated vertex, $f$ is a total domatic coloring of $G\cart H$ with $r$ colors and thus $d_t(G\cart H)\geq r=\max\{d(G),d(H)\}$.
\end{proof}

\noindent If at least one of these graphs $G$ and $H$ is disconnected, then the upper bound can be improved by considering the smallest size of the components of $G$ and $H$.

The bounds given in Theorem \ref{cccartub} are tight for the graphs mentioned in Corollary \ref{cccartGKn}.

\begin{cor}\label{cccartGKn}
Let $ m,n$ be two integers greater than $1$ and $G$  be a graph of order $m$ without an isolated vertex. If $m\leq n$, then $d_t(G\cart K_n)=d(G\cart K_n)=n$. In particular, $d_t(K_m\cart K_n)=d(K_m\cart K_n)=\max\{m,n\}$.
\end{cor}

\begin{proof}
By Theorem \ref{cccartub}, we have $n=\max\{d(G),d(K_n)\}\leq d_t(G\cart K_n)\leq d(G\cart K_n)\leq \max\{m,n\}=n$.
\end{proof}

The tightness of the lower bound for $d_t(G\cart H)$ in Theorem \ref{cccartub} can be also seen by considering $G \cong K_n$ and $H \cong K_2$. More generally, in all cases when $G$ is domatically full and $\delta(H) = 1$, the lower bound is attained for $d_t(G\cart H)$.

\begin{cor}\label{cccart1cycle}
If $G$ is a domatically full graph and $H$ is a graph with minimum degree $1$, then $d_t(G \cart H)=d(G)$.
\end{cor}

\begin{proof}
We have $d_t(G \cart H) \geq d(G)$ by Theorem \ref{cccartub}. The upper bound follows since $d_t(G \cart H) \leq \delta(G \cart H) = \delta(G) + 1 = d(G)$.
\end{proof}

The tightness of the upper bound in Theorem \ref{cccartub} can also be seen by considering $G \cong K_n$ and $H \cong K_2$. Theorem \ref{cccartGH} demonstrates the same in a more general case.

\begin{thm}\label{cccartGH}
Let $r, s_0,s_1,\ldots,s_{r-1}$ be positive integers such that $r\geq2$, $s_0\leq s_1\leq \cdots \leq s_{r-1}$. If $G$ is a graph with total domatic number at least $s_0$ and $H$ is a graph which contains $K_{s_0,s_1,\ldots,s_{r-1}}$ as a spanning subgraph, then $d_t(G\cart H)\geq rs_0$. If each $s_i$ is equal to $s_0$ and $|V(G)|\leq rs_0$,
then $d_t(G\cart H)=|V(H)|=rs_0$ and $d_t(H\cart H)=|V(H)|$. If $G$ is a graph with domatic number at least $s_0$, then the same results hold for $d(G\cart H)$.
\end{thm}

\begin{proof}
Let $G$ be a graph with total domatic number at least $s_0$ and $H$ be a graph which contains the spanning subgraph $H'$, namely $K_{s_0,s_1,\ldots,s_{r-1}}$.
Let $U_i$, $i\in \mathbb{Z}_{s_0}$ be the color classes corresponding to $s_0$-total domatic coloring of $G$.
We label the vertices in each color class $U_i$ of $G$ as $\{u_{ij}:  j\in\mathbb{Z}_{|U_i|}\}$
 and label the vertices of $k^{th}$ part of $H'$, $k\in\mathbb{Z}_{r}$ as $\{v_{kl}:  l\in\mathbb{Z}_{s_k}\}$.
Let us define a coloring $f$ for the vertices of $G\cart H'$ in the following way:
\begin{center}
$f((u_{ij},v_{kl}))=(r(i+l)+k)\pmod{rs_0}.$
\end{center}
The vertex $(u_{ij},v_{kl})$ should be adjacent to the vertices $\{(u_{i' j'},v_{kl}): i'\in\mathbb{Z}_{s_0},\mathrm{~for~ some~} j'\in\mathbb{Z}_{|U_{i'}|}\} $ in the layer $G_{v_{kl}}$ and $\{(u_{ij},v_{k'l'}): k'\in\mathbb{Z}_ r, k'\neq k,  l' \in\mathbb{Z}_{s_{k'}}\}$ in the layer $H_{u_{ij}}$.
Note that $\{(r(i+l')+k')\pmod{rs_0}:k' \in\mathbb{Z}_ r,  l' \in\mathbb{Z}_{s_{k'}}\} =\mathbb{Z}_{rs_0}$.
The set of colors in the neighbors of $(u_{ij},v_{kl})$ are  $\{\{(r(i'+l)+k)\pmod{rs_0}: i'\in\mathbb{Z}_{s_0}\} \cup \mathbb{Z}_{rs_0}\setminus\{(r(i+l')+k)\pmod{rs_0}:  l'\in\mathbb{Z}_{s_{k}}\} \}=\mathbb{Z}_{rs_0}$ as $s_k\geq s_0$.
%
Thus, each vertex of $G\cart H'$ sees all the colors $\mathbb{Z}_{rs_0}$ in its open neighborhood. Since $G\cart H'$ is a spanning subgraph of $G\cart H$, we have $d_t(G\cart H)\geq d_t(G\cart H')\geq rs_0$.

For $r\geq2$, $ i\in\mathbb{Z}_{r}$, if each $s_i$ is equal to $s_0$, then the coloring $f$ mentioned above yields that $d_t(G\cart H)\geq d_t(G\cart H')\geq rs_0=|V(H)|$ and by Theorem \ref{cccartub}, we have $d_t(G\cart H)\leq \max\{|V(G)|,|V(H)|\}=|V(H)|$. Since $d_t(H)\geq s_0$, by above arguments we get $d_t(H\cart H)=|V(H)|$.

If $G$ is a graph with domatic number at least $s_0$, then the same proof mentioned above will work for $d(G\cart H)$ in such a way that
each vertex of $G\cart H'$ should contain the vertices of all the colors $\mathbb{Z}_{rs_0}$ in its closed neighborhood. Thus  $d(G\cart H)\geq d(G\cart H')\geq rs_0$. Also, $d(G\cart H)=|V(H)|=rs_0$ when each $s_i$ is equal to $s_0$ and $|V(G)|\leq rs_0$.
\end{proof}

\begin{rem}
A graph $H$ contains $K_{s_0,s_1, \ldots, s_{r-1}}$, $s_0 \leq s_1\leq \cdots \leq s_{r-1}$, as a spanning graph if and only if the components of $\overline{H}$ can be grouped into $r$ parts such that each part has at least $s_0$ vertices.
\end{rem}

Next, let us consider the Cartesian product of complete graphs and cycles. The upper bound given in Theorem \ref{cccartub} is not tight for cycles of length larger than the size of the complete graphs.

\begin{prop}
Let $ m,n$ be two integers such that $m>n\geq3$, we have $d_t(C_m\cart K_n)=n$.
\end{prop}

\begin{proof}
Let $U=\{u_i:i\in\mathbb{Z}_{m}\}$ and $V=\{v_j:j\in\mathbb{Z}_{n}\}$ be the vertices of $C_m$ and $K_n$ respectively. Suppose there exists an $(n+1)$-total domatic coloring for $C_m\cart K_n$. Since $C_m\cart K_n$ is $(n+1)$-regular, the neighbors of each vertex should be colored distinctly and any edge of the same color class cannot be in the layer $\{u_i\}\cart K_n$ for all $i$, $0\leq i\leq m-1$. Let us consider an edge $(u_i,v_0)(u_{i+1},v_0)$ in $C_m\cart \{v_0\}$ such that both the vertices are colored $c_1$.
For all $x\in \{u_{i-1},u_i,u_{i+1},u_{i+2}\}$, the vertices of $\{x\}\cart K_n$, $(u_{i-2},v_0)$ and $(u_{i+3},v_0)$ cannot be colored with color $c_1$, otherwise there exists a vertex which sees the same color $c_1$ in its two neighbors.
If we choose the another edge  for the color class $c_1$ either $(u_{i+3},y)(u_{i+4},y)$ or $(u_{i-2},y)(u_{i-3},y)$ for some $y\in\{v_1,\ldots,v_{n-1}\}$, then it covers the vertices of exactly $3$ new layers $\{z\}\cart K_n$ for all $z\in \{u_{i+3},u_{i+4},u_{i+5}\}$ or $z\in \{u_{i-2},u_{i-3},u_{i-4}\}$ respectively.
Otherwise, an edge will covers $4$ new layers of $K_n$ in $C_m\cart K_n$. Thus an edge which chosen first has covered $4$ layers of $K_n$ and the subsequent edges covers either $3$ or $4$ layers of $K_n$, likewise we can choose at most $\left\lfloor \frac{m-1}{3}\right\rfloor$ edges for a color class. By (i) of Proposition \ref{DiviReg}, size of an each color class equals $\frac{mn}{n+1}\leq 2\left\lfloor \frac{m-1}{3}\right\rfloor\leq \frac{2m}{3}$
 which in turns $n\leq\frac{2(n+1)}{3}$ and yields $n\leq2$, a contradiction. Hence $d_t(C_m\cart K_n)\leq n$ and lower bound follows from $d_t(C_m\cart K_n)\geq\max\{d(C_m),d(K_n)\}= n$.
\end{proof}

Now, let us consider the bounds for the domatic and total domatic numbers of $G\cart K_2$.

\begin{prop}\label{propnew}
Let $G$ be a graph without an isolated vertex, we have $d(G)\leq d_t(G\cart K_2)\leq2d_t(G)+1\leq 2d(G)+1$ and $d(G\cart K_2)\leq 2d(G)+1$.
\end{prop}

\begin{proof}
Let us consider the graph $G\cart K_2$. By Theorem \ref{cccartub}, $d_t(G\cart K_2)\geq \max\{d(G),2\}=d(G)$.  Let $k=d_t(G\cart K_2)$ and let $v_1,v_2$ be the vertices of $K_2$. Let $D_1,D_2,\ldots, D_{k}$ be a total domatic partition of $V(G\cart K_2)$. It is easy to observe that for any two $i,j\in[k]$, the set $\{u:(u,v)\in D_i\cup D_j\}$ is a total dominating set of $G$.
Thus $d(G)\geq d_t(G)\geq \left\lfloor\frac{k}{2}\right\rfloor$ which in turns that $k=d_t(G\cart K_2)\leq2d_t(G)+1\leq 2d(G)+1$.
Similar arguments hold for the domatic  partition of $V(G\cart K_2)$ and we get $d(G\cart K_2)\leq 2d(G)+1$.
\end{proof}

\begin{figure}
  \centering
  \includegraphics[width=8cm]{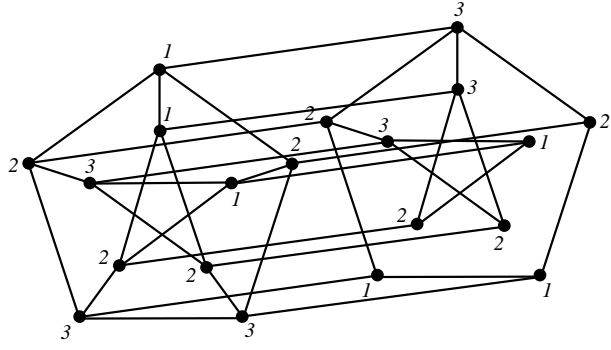}
  \caption{$3$-total domatic coloring of $G\cart K_2$, where $G$ is a Peterson graph}\label{peterk2}
\end{figure}

\noindent There is an example of graph $G$ such that $d_t(G\cart K_2)>d(G)$. Let $G$ be a Peterson graph, we have $d_t(G)=d(G)=2$ and $d_t(G\cart K_2)=3$. The $3$-total domatic coloring of $G\cart K_2$ has been shown in Figure \ref{peterk2}. By (iii) of Proposition \ref{DiviReg}, $4$-total domatic coloring is not possible as $4^2\nmid 40$, where $|E(G\cart K_2)|=40$.
The upper bound $2d_t(G)+1$ for $d_t(G\cart K_2)$ is tight for graphs $K_{n}$,  $n$ odd and $C_{3k}$, $4\nmid k$. We could not find a graph $G$ such that $d_t(G\cart K_2)=2d(G)$. This leads us to ask the following question.

\begin{prob}
Let $G$ be a graph without an isolated vertex, find the smallest constant $c$ such that $d_t(G\cart K_2)\leq cd(G)+O(1)$.
\end{prob}

Let $G$ be the complement of a perfect matching. It is easy to observe that $d(G)=\frac{|V(G)|}{2}$ and $d(G\cart K_2)=|V(G)|=2d(G)$.
Also, we could not find a graph $G$ such that $d(G\cart K_2)=2d(G)+1$. This leads us to ask a question: Is there any graph $G$ such that $d(G\cart K_2)>2d(G)$?

Now, let us consider a lower bound for the domatic and total domatic numbers of Cartesian product of a graph and a bipartite graph in terms of its domatic and total domatic numbers. The bound given in Theorem \ref{ccbicart} has been applied multiple times in this paper.

\begin{thm}\label{ccbicart}
Let $G$ be a graph and $H$ be a bipartite graph, we have
\begin{enumerate}[(i)]
\setlength\itemsep{-2pt}
\item $ d_t(G\cart H)\geq 2\min\{d_t(G),d_t(H)\}$ and
\item  $d(G\cart H)\geq 2\min\{d(G),d_t(H)\}$.
\end{enumerate}
\end{thm}

\begin{proof}
(i) Let $G$ be a graph of order $m$ and $H$ be a bipartite graph of order $n$ with bipartition $[X,Y]$.
Let $\{u_i:i\in\mathbb{Z}_m\}$ and $\{v_j:j\in\mathbb{Z}_n\}$ be the vertices of $G$ and $H$ respectively. Let $k=\min\{d_t(G),d_t(H)\}$, there exists a total domatic coloring for $G$ and $H$ with $k$ colors.
Let $g$ and $h$ be a $k$-total domatic coloring of $G$ and $H$ respectively, and let $\mathbb{Z}_{k}$ be the $k$ colors. Now, let us define a coloring $f$ for the vertices of $G\cart H$. For $i\in\mathbb{Z}_m$ and $j\in\mathbb{Z}_n$,
 \begin{equation*}
 f((u_i,v_j))=\left\{\begin{array}{ll}
                                                      (2 g(u_i)+2h(v_j))\pmod{2k}   & \mathrm{if}~v_j\in X \\
                                                     (2 g(u_i)+2h(v_j)+1)\pmod{2k}   & \mathrm{if}~v_j\in Y.
                                                     \end{array}
                                                     \right. \eqno{(3.1)}
 \end{equation*}
For any vertex $u_i\in V(G)$ and $v_j\in X$, by the coloring $f$, the colors seen by the vertex $(u_i,v_j)$ from the neighbors of $u_i$ in $G_{v_j}$ are $\{(2s+2h(v_j))\pmod{2k}:s\in\mathbb{Z}_{k}\}=\{2l:l\in \mathbb{Z}_{k}\}$ and from the neighbors of  $v_j$ in $H_{u_i}$ are $\{(2g(u_i)+2t+1)\pmod{2k}:t\in\mathbb{Z}_{k}\}=\{2l+1:l\in\mathbb{Z}_{k}\}$.
Similarly, for any vertex $u_i\in V(G)$ and $v_j\in Y$, the colors seen by the vertex $(u_i,v_j)$ from the neighbors of $u_i$ in $G_{v_j}$ are  $\{2l+1:l\in\mathbb{Z}_{k}\}$ and from the neighbors of $v_j$ in $H_{u_i}$ are $\{2l:l\in\mathbb{Z}_{k}\}$.
Each vertex $(u_i,v_j)$ in $G\cart H$ sees all the colors $\mathbb{Z}_{2k}$ in its open neighborhood and thus $f$ is a total domatic coloring using $2k$ colors. Hence $d_t(G\cart H)\geq2k=2\min\{d_t(G),d_t(H)\}$.

(ii) Let $k=\min\{d(G),d_t(H)\}$. Let $g$ be a $k$-domatic coloring of $G$ and $h$ be a $k$-total domatic coloring of $H$.
For any vertex $u_i\in V(G)$ and $v_j\in X$, by the coloring $f$ defined in Equation (3.1), the colors seen by the vertex $(u_i,v_j)$ from the closed neighborhood of $u_i$ in $G_{v_j}$ are $\{(2s+2h(v_j))\pmod{2k}:s\in\mathbb{Z}_{k}\}=\{2l:l\in \mathbb{Z}_{k}\}$ and from the open neighborhood of $v_j$ in $H_{u_i}$ are $\{(2g(u_i)+2t+1)\pmod{2k}:t\in\mathbb{Z}_{k}\}=\{2l+1:l\in\mathbb{Z}_{k}\}$.
Similarly, for any vertex $u_i\in V(G)$ and $v_j\in Y$, the colors seen by the vertex $(u_i,v_j)$ from the closed neighborhood of $u_i$ in $G_{v_j}$ are $\{2l+1:l\in\mathbb{Z}_{k}\}$ and from the open neighborhood of $v_j$ in $H_{u_i}$ are $\{2l:l\in\mathbb{Z}_{k}\}$.
Each vertex $(u_i,v_j)$ in $G\cart H$ contains vertices of all the colors $\mathbb{Z}_{2k}$ in its closed neighborhood and thus $f$ is a domatic coloring using $2k$ colors. Hence $d(G\cart H)\geq 2\min\{d(G),d_t(H)\}$.

If at least one of these graphs $G$, $H$ is disconnected, then apply the same technique to each component of $G\cart H$ separately.
\end{proof}

The bound given in Theorem \ref{ccbicart} for $d(G\cart H)$ is tight, which follows by taking $G\cong K_{n}$ and $H\cong K_{n,n}$ and
also, the bound for $d_t(G\cart H)$ is tight by taking $G\cong K_{2n}$ and $H\cong K_{n,n}$.
One of the simplest examples such that a strict inequality holds in Theorem \ref{ccbicart} is $G\cong C_{6n}$, when $n$ is odd. By Theorem \ref{cccartub}, we have $d_t(C_{6n}\cart C_{6n})\geq d(C_{6n})=3$ but $d_t(C_{6n})=1$.
Theorem \ref{ccbicart} is not true for all graphs in general. For $G\cong K_2\cart K_3$, we have $d_t(G)=3$.
Since $G\cart G$ is $6$-regular,
the neighbors of each vertex should be colored distinctly when $d_t(G\cart G)=6$. There exists a color which occurs at least twice in the subgraph $K_3\cart K_3$ of $G\cart G$ and
hence there exists a vertex which contains the same color in its two neighbors, a contradiction.
Thus $d_t(G\cart G)< 6=2d_t(G)$.
We can extend the
above idea to show that there exists graph $G$ such that
$d_t(G \cart G)\leq d(G \cart G) \leq d_t(G) + O\left(\sqrt{d_t(G)}\right)$.

\begin{prop}
For the graph $G = K_n \cart K_2$,
$d_t(G \cart G) \leq d(G\cart G) \leq d_t(G) + \sqrt{2d_t(G)}$.
\end{prop}

\begin{proof}\label{cccartKnK2KnK2}

From Corollary~\ref{cccartGKn} we know that $d_t(G) =d(G)= n$ and we will show
that $d(G \cart G) \leq n + \sqrt{2n}$. Suppose that there exists a
domatic coloring $f$ of $G \cart G$ using more than $n + \sqrt{2n}$ colors,
then there exists a color (call it red) which appears in at most $4(n -
\sqrt{2n})$ vertices of $G \cart G$.  Otherwise, we get a contradiction, since
$4(n - \sqrt{2n} + 1)(n + \sqrt{2n} + 1) = 4(n^2 + 1) > |V(G \cart G)|$.

Observe that $G \cart G$ is isomorphic to $H \cart C_4$, where $H = K_n
\cart K_n$. We label the vertices of $C_4$ with $\mathbb{Z}_4$ and the
copy of $H$ corresponding to vertex $i$ of this $C_4$ will be called
$H_i$. Let $R_i$ denote the set of vertices in $H_i$ which are colored red by
$f$. Since $|R_0 \cup \cdots \cup R_3| \leq 4(n- \sqrt{2n})$, the smallest
of them, without loss of generality say $R_0$, has at most $n - \sqrt{2n}$
vertices. Let $k = n - |R_0|$, and note that $k \geq \sqrt{2n}$.

The coloring $f$ on $H_0$ can be represented by an $n \times n$ matrix $M_0$.
Since $|R_0| = n - k$, there exists at least $k$ rows and $k$ columns of $M_0$
which do not contain any red vertex. Hence the $k^2$ vertices corresponding to
the $k \times k$ submatrix determined by the above rows and columns are neither red, nor they see
a red neighbor in $H_0$.  Hence each of them have a red neighbor in $H_1$ or
$H_3$. Since each vertex in $V(H_1) \cup V(H_3)$ is adjacent to exactly one
vertex in $V(H_0)$, we can conclude that $|R_1 \cup R_3| \geq k^2$. But now,
since $|R_2| \geq |R_0|$, we get $|R_0 \cup \cdots \cup R_3| \geq 2(n-k) +
k^2$. Since $k^2 - 2k$ is an increasing function of $k$ for $k \geq 1$, the
above lower bound is at least $4n - 2\sqrt{2n}$ for any $k \geq \sqrt{2n}$.
This contradicts our earlier upper bound on $|R_0 \cup \cdots \cup R_3|$.
\end{proof}

One of the consequences of Theorem \ref{ccbicart}  is Corollary \ref{ccbicarteq}. 

\begin{cor} \label{ccbicarteq}
Let $d$ be a positive integer, $n$ and $k$ be powers of $2$, $k\leq n$. If $G$ is a graph with $d(G)\geq kd$ and for $1\leq i \leq n$, $H_i$ is a bipartite graph such that $d_t(H_i)\geq d$,  then $d_t(\mathop{\cart}\limits_{i=1}^{n}H_i)\geq nd$ and $d(G\cart(\mathop{\cart}\limits_{i=1}^{n-k}H_i))\geq nd$.
\end{cor}

\begin{proof}
For each $j$ taken in the increasing order $1,2,\ldots,\frac{n}{2}$, repeated application of Theorem \ref{ccbicart} to pairs of bipartite graphs that have a total domatic coloring with $jd$ colors yields the bipartite graphs having a total domatic coloring with $2jd$ colors. Finally, we get $d_t(\mathop{\cart}\limits_{i=1}^{n}H_i)\geq nd$. Now, let $G$ be a graph  such that $d(G)\geq kd$.
First we group $H_1,H_2,\ldots,H_{n-k}$ into groups of size $k$ each
and apply the above argument to get the bipartite graphs $H_1',H_2',\ldots,H_{\frac{n}{k}-1}'$ with $d_t(H_j')\geq kd$, $1\leq j\leq \frac{n}{k}-1$.
Now, consider the graphs in this order $G,H_1',H_2',\ldots,H_{\frac{n}{k}-1}'$. Apply Theorem \ref{ccbicart} (ii) to the first pair and Theorem \ref{ccbicart} (i) to the remaining pairs of graphs repeatedly as mentioned above to obtain the final result.
Thus $d(G\cart(\mathop{\cart}\limits_{i=1}^{n-k}H_i))\geq (\frac{n}{k})kd=nd$.
\end{proof}

\section{Hamming graphs}

Let us start this section by re-derive the results  for the domatic number of the hypercubes $Q_{n-1}$ and $Q_n$ \cite{zel4} and the total domatic number of the hypercube $Q_n$ \cite{chen}, when $n$ is a power of $2$.

\begin{cor} \label{ccbicarteqQn}
Let $n$ be a powers of $2$, we have $d_t(Q_{n})=d_t(Q_{n+1})=n$ and $d(Q_{n-1})=d(Q_{n})=n$.
\end{cor}

\begin{proof}
It is easy to see that $d_t(Q_n)\leq n$ and $d(Q_{n-1})\leq n$. Since $d_t(K_2)=1$, $d(K_2)=2$, $Q_{n}\cong\mathop{\cart}\limits_{i=1}^{n} K_2$ and $Q_{n-1}\cong K_2\cart(\mathop{\cart}\limits_{i=1}^{n-2} K_2)$, by Corollary \ref{ccbicarteq}, we get $d_t(Q_{n})\geq n$ and
 $d(Q_{n-1})\geq n$.
Suppose $d_t(Q_{n+1})=n+1$, by (iii) of Proposition \ref{DiviReg}, $n+1$ should divides $2^{n+1}$, a contradiction. Thus $d_t(Q_{n+1})\leq n$.
Also,  $d_t(Q_{n+1})\geq d_t(Q_{n})= n $.
Similarly, by Theorem \ref{DiviRegDom}, we get $d(Q_n)\leq n$ and $d(Q_{n})\geq d(Q_{n-1})= n $.
\end{proof}

In \cite{hav}, Havel obtained that $d(Q_6)=5$. Also, $5=d(Q_6)\leq d_t(Q_7)\leq \left\lfloor\frac{2^7}{24}\right\rfloor=5$. 
We obtain a new lower bound for $d(Q_n)$ and $d_t(Q_n)$ for some values of $n$ in Corollary \ref{domqn}.

\begin{cor}\label{domqn}
Let $k,n$ be positive integers such that $ n\geq 2^k7-1$, we have $d(Q_n)\geq 2^k 5$ and $d_t(Q_{n+1})\geq 2^k 5$
\end{cor}

\begin{proof}
Let $G\cong Q_{6}$ and $H_i\cong Q_{7}$, for $1\leq i\leq2^k$.
Clearly,  $d(G)=5$ and $d_t(H_i)=5$.
By Corollary \ref{ccbicarteq}, we get $d(G\cart(\mathop{\cart}\limits_{i=1}^{2^k-1}H_i))\geq 2^k5$ and $d_t(\mathop{\cart}\limits_{i=1}^{2^k}H_i)\geq 2^k5$.
Since  $n\geq 2^k7-1$, we have $d(Q_n)\geq 2^k 5$ and $d_t(Q_{n+1})\geq 2^k 5$.
\end{proof}

When $n$ is a power of $2$, we can also re-derive the following results on the domination and total domination numbers of the hypercubes $Q_{n-1}$ \cite{har} and $Q_n$ \cite{joh} respectively.

\begin{cor} \cite{har,joh}\label{ccbicarteqgam}
For a positive integer $k$, $\gamma(Q_{2^k-1})=2^{2^k-k-1}$, $\gamma_t(Q_{2^k})=2^{2^k-k}$, $\gamma(Q_{2^k})\leq2^{2^k-k}$ and $\gamma_t(Q_{2^k+1})\leq2^{2^k-k+1}$.
\end{cor}

\begin{proof}
By Corollary \ref{domcor}, \ref{ccbicarteqQn} and (iv) of Proposition \ref{DiviReg}, we get $\gamma(Q_{2^k-1})=\frac{2^{2^k-1}}{2^k}=2^{2^k-k-1}$ and $\gamma_t(Q_{2^k})=\frac{2^{2^k}}{2^k}=2^{2^k-k}$ respectively. Also, by Corollary \ref{ccbicarteqQn}, we have $\gamma(Q_{2^k})\leq \frac{2^{2^k}}{2^k}=2^{2^k-k}$ and $\gamma_t(Q_{2^k+1})\leq \frac{2^{2^k+1}}{2^k}=2^{2^k-k+1}$.
\end{proof}

Note that $\gamma(Q_{2^k})=2^{2^k-k}$ follows from the sphere bound mentioned in \cite{van} and
$\gamma_t(Q_{2^k+1})=2^{2^k-k+1}$ follows from the result proved by Azarija et al., \cite{aza} namely, $\gamma_t(Q_{n+1})=2\gamma(Q_n)$.

In Table \ref{tab1}, we have mentioned the present best bounds for $\gamma(Q_n)$ (see, Table $1$ in \cite{ber,van}) and $\gamma_t(Q_n)$ follows from the result $\gamma_t(Q_{n})=2\gamma(Q_{n-1})$ \cite{aza}. The bounds for $d(Q_n)$ and $d_t(Q_n)$ follows from Corollary \ref{ccbicarteqQn} (also see, \cite{chen,zel4}) and Corollary \ref{domqn}.


\begin{table}
  \centering
   \vskip.3cm
   \footnotesize
     \begin{tabular}{|c|c|c|c|c|}
     \hline
    $n$ & $\gamma(Q_n)$ & $\gamma_t(Q_n)$ & $d(Q_n)$ & $d_t(Q_n)$ \\ \hline
     $1$ & $1$ & $2$ & $2$ & $1$ \\
      $2$& $2$ & $2$ & $2$ & $2$ \\
     $3$ & $2$& $4$ & $4$ & $2$ \\
     $4$& $4$ & $4$ & $4$ & $4$\\
     $5$& $7$& $8$ & $4$ & $4$ \\
     $6$ & $12$& $14$ & $5$ & $4$ \\
     $7$ & $16$& $24$ & $8$ & $5$ \\
    $8$ & $32$& $32$ & $8$ & $8$ \\
     $9$& $62$& $64$ & $8$ & $8$ \\
     $10$& $107$-$120$& $124$ & $8$-$9$ & $8$ \\
     $11$& $180$-$192$& $214$-$240$ & $8$-$11$ & $8$-$9$ \\
     $12$& $342$-$380$& $360$-$384$ & $8$-$11$ & $8$-$11$ \\
     $13$& $598$-$704$ & $684$-$760$ & $10$-$13$ & $8$-$11$ \\
     $14$& $1171$-$1408$ & $1196$-$1408$& $10$-$13$ & $10$-$13$ \\
     $15$& $2^{11}$ & $2342$-$2816$& $16$ & $10$-$13$ \\
     $16$& $2^{12}$ & $2^{12}$ & $16$ & $16$\\
     $17$&$7377$-$2^{13}$ & $2^{13}$ & $16$-$17$ & $16$ \\
     \hline
   \end{tabular}
\caption{Bounds on $\gamma(Q_n)$, $\gamma_t(Q_n)$, $d(Q_n)$ and $d_t(Q_n)$, $1\leq n\leq 17$. Lower bound for $d(Q_{13})$, $d(Q_{14})$, $d_t(Q_{14})$ and $d_t(Q_{15})$ are new and follow from Corollary \ref{domqn}} \label{tab1}
   \end{table}

\normalsize

%

Now, let us start obtain a lower bound for the domatic and total domatic numbers of Hamming graphs $H_{n-1,q}$ and $H_{n,q}$ respectively when $n$ is power of $2$ and $q\geq2$.

\begin{thm}\label{H2r2l}
Let $q$ be an integer greater than $1$ and $n$ is a power of $2$, we have $d(H_{n-1,q})\geq n\left\lfloor \frac{q}{2}\right\rfloor$ and $d_t(H_{n,q})\geq n\left\lfloor \frac{q}{2}\right\rfloor$.
\end{thm}

\begin{proof}
Let $K_{\left\lceil \frac{q}{2}\right\rceil, \left\lfloor \frac{q}{2}\right\rfloor}$ be a complete bipartite subgraph of $K_q$.
Any coloring which assigns same set of $\left\lfloor \frac{q}{2}\right\rfloor$ different colors to each part of $K_{\left\lceil \frac{q}{2}\right\rceil, \left\lfloor \frac{q}{2}\right\rfloor}$ is a total domatic coloring with  $\left\lfloor \frac{q}{2}\right\rfloor$ colors.
Since $n$ is a power of $2$, by Corollary \ref{ccbicarteq}, we have $d_t(H_{n,q})=d_t(\mathop{\cart}\limits_{i=1}^{n}K_{q})\geq d_t(\mathop{\cart}\limits_{i=1}^{n}K_{\left\lceil \frac{q}{2}\right\rceil, \left\lfloor \frac{q}{2}\right\rfloor})= n\left\lfloor\frac{q}{2}\right\rfloor$. Also, $d(K_q)=q\geq 2\left\lfloor\frac{q}{2}\right\rfloor$, by Corollary \ref{ccbicarteq}, we have $d(H_{n-1,q})=d(\mathop{\cart}\limits_{i=1}^{n-1}K_{q})\geq d(K_q\cart(\mathop{\cart}\limits_{i=1}^{n-2}K_{\left\lceil \frac{q}{2}\right\rceil, \left\lfloor \frac{q}{2}\right\rfloor}))\geq n\left\lfloor\frac{q}{2}\right\rfloor$.
\end{proof}

Note that an equality holds for the graphs $H_{1,q}$, $H_{2,2q}$, and the problem remains open for $H_{n,q}$, $n\geq4$ except $H_{n,2}$.

In \cite{chen} Chen et al., obtained the injective chromatic number of $H_{n,q}$ as follows.

\begin{thm}\cite{chen}\label{chiiHnq}
Let  $k,n$  be positive integers. If $q$ is a prime power and $n= \frac{(q^k-1)}{(q-1)}$, then $\chi_i(H_{n,q})=\chi(H_{n,q}^2)=q^k$.
\end{thm}

Corollary \ref{chicHnq} is an immediate consequence of Theorem \ref{cccartub} and Theorem \ref{chiiHnq}.

\begin{cor}\label{chicHnq}
Let  $k,n$  be positive integers. If $q$ is a prime power and $n= \frac{(q^k-1)}{(q-1)}$,  then $d(H_{n,q})=q^k$ and $d_t(H_{n+1,q})\geq q^k$.
\end{cor}

\begin{proof}
By Theorem \ref{chiiHnq}, we have $\chi_i(H_{n,q})=\chi(H^{2}_{n,q})=q^k=n(q-1)+1$.
Since $H_{n,q}$ is $n(q-1)$-regular and there is a proper injective coloring with $n(q-1)+1$ colors, by Corollary \ref{domcor} we get $d(H_{n,q})=q^k$.  By Theorem \ref{cccartub}, we have $d_t(H_{n+1,q})=d_t(H_{n,q}\cart K_q)\geq d(H_{n,q})=q^k$.
\end{proof}

As a consequence of Theorem \ref{cccartub}, \ref{H2r2l} and Corollary \ref{ccbicarteq}, \ref{chicHnq}, we get Corollary \ref{cccartHnq}

\begin{cor} \label{cccartHnq}
 Let $n$ be a positive integer and $q$ be a prime power. If $k$ is the largest positive integer such that  $\frac{q^{k}-1}{q-1}\leq n$, $j$ is the smallest positive integer such that $q^k\leq2^j\left\lfloor\frac{q}{2}\right\rfloor$ and $i$ is  the largest positive integer such that $\frac{q^{k}-1}{q-1}+(2^{i}-1)2^j\leq n$, then $d(H_{n,q})\geq2^{i} q^k$ and $d_t(H_{n+1,q})\geq2^{i} q^k$.
\end{cor}

\begin{proof}
Let $n'=\frac{q^{k}-1}{q-1}$ and $G\cong H_{n',q}$. For $0\leq l<2^i$, let $G_l\cong H_{2^j,q}$.
It is clear from the proof of Theorem \ref{H2r2l}, there exists a spanning bipartite subgraph of $G_l$ that have a total domatic coloring with $2^{j}\left\lfloor\frac{q}{2}\right\rfloor$ colors, let it be $H_l$.
Clearly,  $d_t(H_l)\geq2^{j}\left\lfloor\frac{q}{2}\right\rfloor\geq q^k$ and by Corollary \ref{chicHnq}, $d(G)=q^k$.
By Corollary \ref{ccbicarteq}, we get $d(G\cart(\mathop{\cart}\limits_{l=1}^{2^i-1}H_l))\geq 2^i q^k$.
Since  $n\geq n'+(2^{i}-1)2^j$,  we have $d(H_{n,q})\geq d(H_{n',q}\cart(\mathop{\cart}\limits_{l=1}^{2^i-1}G_l))\geq d(G\cart(\mathop{\cart}\limits_{l=1}^{2^i-1}H_l))\geq 2^i q^k$.
By Theorem \ref{cccartub}, we have $d_t(H_{n+1,q})\geq d(H_{n,q})\geq2^i q^k$.
\end{proof}

\section{Tori}

In this section, we examine the $d$-dimensional tori which are domatically and total domatically full.
First we obtain some sufficient conditions for the tori which are total domatically full.

Corollary \ref{cccartcycle} as an immediate consequence of Corollary \ref{ccbicarteq}
 and the fact that $d_t(C_n)=2$ if and only if  $n\equiv 0\pmod4$.

\begin{cor}\label{cccartcycle}
Let $d, k_1,k_2,\ldots,k_{d}$ be positive integers such that $d$ is a power of $2$ and $k_i\equiv 0\pmod4$, $1\leq i\leq d$, we have $d_t(\mathop{\cart}\limits_{i=1}^{d}C_{k_i})=2d$ and $\gamma_t(\mathop{\cart}\limits_{i=1}^{d} C_{k_i})=({\mathop\prod\limits_{i=1}^{d} k_i})/{2d}$.
\end{cor}

In the remaining part of this section, we try to generalize this result to larger collections of tori.

In \cite{grav2}, Gravier independently obtained the total domination number of some tori by the concept of periodic tiling which mentioned in \cite{grav1}.

\begin{thm}\cite{grav2}\label{tdcartcycled}
Let $d, k_1,k_2,\ldots, k_{d}$ be positive integers such that $d\geq2$ and $k_i\equiv 0\pmod{4d}$, for $1\leq i\leq d$, we have
$\gamma_t(\mathop{\cart}\limits_{i=1}^{d} C_{k_i})=({\mathop\prod\limits_{i=1}^{d} k_i})/{2d}.$
Moreover, if $d$ is even and for any positive integer $k_i$, $1\leq i\leq d$ such that $k_i\equiv 0\pmod{2d}$,
then this equality still holds.
 \end{thm}

By vertex transitivity of the tori, one can obtain a total domatic coloring of the tori mentioned in Theorem \ref{tdcartcycled} with $2d$ colors. Hence this tori are total domatically full.
We extend this to larger classes of tori. Moreover, our proof is much shorter.

\begin{thm}\label{cccartcycle2d}
Let $d, k_1,k_2,\ldots, k_d$  be positive integers. If  $k_d$ is congruent to $0\pmod{4}$  and the remaining $k_i$'s are congruent to $0\pmod{2d}$, then $d_t(\mathop{\cart}\limits_{i=1}^d C_{k_i})=2d$.
\end{thm}

\begin{proof}
Let $G\cong \mathop{\cart}\limits_{i=1}^d C_{k_i}$, where $d$, $k_i$'s are positive integers. Since $\delta(G)=2d$, it is enough to give a $2d$-total domatic coloring for $G$. Let the vertices of $G$ be $\{x=(x_1,x_2,\ldots,x_{d}): x_i\in\mathbb{Z}_{k_i}, 1\leq i\leq d\}$.
Now, let us define a coloring $f$ by
 \begin{center}
 $f(x)=\left\{\begin{array}{ll}
                                      \sum_{i=1}^{d-1}i\ x_i\ \pmod{2d}  & \mathrm{if}~ x_d \equiv 0 ~\mathrm{or}~1\pmod4 \\
                                      (\sum_{i=1}^{d-1}i\ x_i)+d\ \pmod{2d}  & \mathrm{if}~   x_d \equiv 2 ~\mathrm{or}~3\pmod4.                                                     \end{array}
                                                     \right. $
 \end{center}
Let $k$ be the color of the vertex $x=(x_1,x_2,\ldots,x_d)$ defined by $f$.
The set of neighbors of $x$ are $\{\{(x_1,\ldots,x_{i-1},x_i-1,x_{i+1},\ldots, x_d), (x_1,\ldots,x_{i-1},x_i+1,x_{i+1},\ldots,x_d)\}: 1\leq i\leq d\}$. The set of colors in the neighbors of $x$ are the union of $\{k-i, k+i \}$ along the dimension $i, 1\leq i\leq d-1$ and $\{k,k+d\}$ along the dimension $d$ which equals $\mathbb{Z}_{2d}$.
Thus each vertex contains vertices of all $2d$ colors $\mathbb{Z}_{2d}$ in its open neighborhood and $f$ is a $2d$-total domatic coloring for $G$. Hence, $d_t(\mathop{\cart}\limits_{i=1}^d C_{k_i})=2d$.
\end{proof}

In Theorem \ref{cccartcycle2d}, $d$-dimensional tori are total domatically full when a cycle of length is a multiple of $4$ and other cycles length are a multiple of $2d$ but this condition can be further generalized.
We obtain a sufficient condition for the total  domatically full tori in Corollary \ref{cccartcycle2dcor} which generalize the results mentioned in Corollary \ref{cccartcycle} and Theorem \ref{cccartcycle2d}.

\begin{cor}\label{cccartcycle2dcor}
Let $d=2^{p}q$, $q$ be an odd integer and $p\geq0$. If $2^{p}$ number of $k_i$'s are congruent to $0\pmod{4}$  and the remaining $k_i$'s are congruent to $0\pmod{2q}$, then $d_t(\mathop{\cart}\limits_{i=1}^d C_{k_i})=2d$.
\end{cor}
\begin{proof}
Let $G\cong \mathop{\cart}\limits_{i=1}^d C_{k_i}$, where $d$, $k_i$'s are positive integers. For  $p=0$, the proof follows from Theorem \ref{cccartcycle2d}.
Let us consider $p\geq1$.
Since $\mathop{\cart}\limits_{i=1}^{d} C_{k_i}$ is transitive with respect to the product, among $2^{p}$ number of $k_i$'s are congruent to $0\pmod{4}$, without loss of generality, choose those $k_i$'s are $k_{q}, k_{2q},\ldots,k_{d}$.
Now, split the graph $\mathop{\cart}\limits_{i=1}^{d} C_{k_i}$ into product of $2^{p}$ smaller product of $q$ graphs each as $(\mathop{\cart}\limits_{i=1}^{q} C_{k_i})\cart(\mathop{\cart}\limits_{i=q+1}^{2q} C_{k_{i}})\cart\cdots\cart(\mathop{\cart}\limits_{i=d-q+1}^{d} C_{k_{i}})\cong\mathop{\cart}\limits_{j=1}^{2^{p}} G_j$, where  $d=2^pq$.
By Theorem \ref{cccartcycle2d}, each smaller product graph $G_j$ has the total domatic number $2q$ for $1\leq j\leq 2^{p}$.
By Corollary \ref{ccbicarteq}, we have $d_t(\mathop{\cart}\limits_{i=1}^{d}C_{k_i})=d_t(\mathop{\cart}\limits_{j=1}^{2^{p}} G_j)=2^{p}2q=2d$.
\end{proof}

The Corollary \ref{tdcartcycle} is an immediate consequence of Corollary \ref{cccartcycle2dcor} and (iv) of Proposition \ref{DiviReg} which turns to be a generalization of the result given in Corollary \ref{cccartcycle} and a result due to S. Gravier \cite{grav2} for the $d$-dimensional tori given in Theorem \ref{tdcartcycled}.

\begin{cor}\label{tdcartcycle}
Let $d=2^{p}q$, $q$ be an odd integer and $p\geq0$. If $2^{p}$ number of $k_i$'s are congruent to $0\pmod{4}$  and the remaining $k_i$'s are congruent to $0\pmod{2q}$, then
$\gamma_t(\mathop{\cart}\limits_{i=1}^d C_{k_i})=({\mathop\prod\limits_{i=1}^{ d} k_i})/{2d}$.
\end{cor}

Note that, the family of tori given in Corollary \ref{tdcartcycle} contains the family of tori as mentioned in Theorem \ref{tdcartcycled}.
Suppose the dimension of the torus is $12$,
Corollary \ref{tdcartcycle} find $\gamma_t$ for the tori $\mathop{\cart}\limits_{i=1}^{12} C_{p_i}$, where four $p_i$'s are congruent to $0\pmod{4}$
 and the remaining $p_i$'s are congruent to $0\pmod{6}$ but Theorem \ref{tdcartcycled} finds $\gamma_t$ for the torus $\mathop{\cart}\limits_{i=1}^{12} C_{k_i}$, where each $k_i\equiv0\pmod{24}$.

Corollary \ref{ccbicartHam} is an immediate consequence of Corollary \ref{cccartcycle2dcor}, \ref{tdcartcycle}.

\begin{cor} \label{ccbicartHam}
Let $d=2^{p}q$, $q$ be an odd integer and $p\geq0$ and let $G_i$ be a Hamiltonian graph, $1\leq i\leq d$. If $2^{p}$ number of $G_i$'s have $|V(G_i)|\equiv 0\pmod{4}$ and the remaining $G_i$'s have $|V(G_i)|\equiv0\pmod{2q}$,
then $d_t(\mathop{\cart}\limits_{i=1}^{d}G_i)\geq2{d}$ and $\gamma_t(\mathop{\cart}\limits_{i=1}^d G_i)\leq({\mathop\prod\limits_{i=1}^{ d} |V(G_i)|})/{2d}$.
\end{cor}

Next, we obtain a sufficient condition for the domatically full tori.
\begin{thm}\label{cccartcycle2d+1}
Let $d, k_1,k_2,\ldots, k_d$  be positive integers. If  each $k_i$ is congruent to $0\pmod{2d+1}$, then $d(\mathop{\cart}\limits_{i=1}^d C_{k_i})=2d+1$.
\end{thm}

\begin{proof}
Let $G\cong \mathop{\cart}\limits_{i=1}^d C_{k_i}$, where $d$, $k_i$'s are positive integers and each $k_i\equiv 0\pmod{2d+1}$. Since $\delta(G)+1=2d+1$, it is enough to give a $(2d+1)$-domatic coloring for $G$. Let the vertices of $G$ be $\{x=(x_1,x_2,\ldots,x_{d}): x_i\in\mathbb{Z}_{k_i}, 1\leq i\leq d\}$.
Now, let us define a coloring $f$ by
 \begin{center}
 $f(x)= \sum_{i=1}^{d}i\ x_i\ \pmod{2d+1}$.
 \end{center}
Let $k$ be the color of the vertex $x=(x_1,x_2,\ldots,x_d)$ defined by $f$.
The set of neighbors of $x$ are $\{\{(x_1,\ldots,x_{i-1},x_i-1,x_{i+1},\ldots, x_d), (x_1,\ldots,x_{i-1},x_i+1,x_{i+1},\ldots,x_d)\}: 1\leq i\leq d\}$.
The set of colors in the neighbors of $x$ are $\{k-i, k+i \}$ along the dimension $i, 1\leq i\leq d$ which equals $\mathbb{Z}_{2d+1}\setminus\{k\}$.
Thus each vertex contains vertices of all $2d+1$ colors $\mathbb{Z}_{2d+1}$ in its closed neighborhood. Hence $f$ is a $(2d+1)$-domatic coloring for $G$ and $d(\mathop{\cart}\limits_{i=1}^d C_{k_i})=2d+1$.
\end{proof}

As a consequence of Theorem \ref{cccartcycle2d+1} and Corollary \ref{domcor}, we get Corollary \ref{domcycle2d+1} which re-derive the result
proved by Klav\v{z}ar and Seifter \cite{sand}.

\begin{cor}\cite{sand}\label{domcycle2d+1}
Let $d, k_1,k_2,\ldots, k_d$  be positive integers. If  each $k_i$ is congruent to $0\pmod{2d+1}$, then $\gamma(\mathop{\cart}\limits_{i=1}^{d} C_{k_i})=({\mathop\prod\limits_{i=1}^{d} k_i})/({2d+1})$.
\end{cor}

Corollary \ref{ccbicartHamdom} is an immediate consequence of Theorem \ref{cccartcycle2d+1} and Corollary \ref{domcycle2d+1}.

\begin{cor} \label{ccbicartHamdom}
Let $d$  be a positive integer and let $G_i$ be a Hamiltonian graph, $1\leq i\leq d$. If  each $G_i$ have $|V(G_i)|\equiv0\pmod{2d+1}$, then $d(\mathop{\cart}\limits_{i=1}^{d}G_i)\geq2{d}+1$ and $\gamma(\mathop{\cart}\limits_{i=1}^d G_i)\leq({\mathop\prod\limits_{i=1}^{ d} |V(G_i)|})/({2d+1})$.
\end{cor}

\section{Conclusion and open problems}

Accepting the invitation by Chen, Kim, Tait and Verstraete \cite{chen} to
determine any relationships between $d_t(G \cart G)$ and $d_t(G)$, we
started this investigation aiming to find good lower and upper bounds to
$d_t(G \cart G)$ in terms of $d_t(G)$. In this paper, we have made
improvements to the easy lower bound $d_t(G \cart G) \geq d_t(G)$.
Firstly we showed that if $\delta(G) \geq 1$, then $d_t(G \cart G) \geq
d(G)$. We also showed
that $d_t(G \cart G) \geq 2d_t(G)$ if $G$ is bipartite.  Bipartiteness
is necessary for the above lower bound. We can show the existence of infinite
families of non-bipartite graphs where $d_t(G \cart G) = d_t(G) +
\sqrt{2d_t(G)}$.
Nevertheless, we have indirectly used this result for non-bipartite graphs in
the form $d_t(G \cart G) \geq 2d_t(G')$, where $G'$ is a spanning
bipartite subgraph of $G$.

In contrast, we haven't been able to prove any upper bound for $d_t(G
\cart G)$ in terms of $d_t(G)$. We know of graphs $G$ ($G = K_{2n+1}$ for
example) where $d_t(G \cart G) = 2d_t(G) + 1$. We conjecture that it is
the maximum possible.

\begin{conj}
\label{conjUpperBound}
For any two graphs $G$ and $H$ without an isolated vertex,
\begin{center}
$d_t(G\cart H)\leq 2\max\{d_t(G),d_t(H)\}+1.$
\end{center}
\end{conj}

A weaker form of the above conjecture, $d_t(G\cart H)\leq
2\max\{d(G),d(H)\}+1$, is also an interesting open problem.  As far as we know,
the version of Conjecture~\ref{conjUpperBound} for domatic number is also open.

\begin{conj}
\label{conjUpperBoundDom}
For any two graphs $G$ and $H$ without an isolated vertex,
\begin{center}
$d(G\cart H)\leq 2\max\{d(G),d(H)\}+1.$
\end{center}
\end{conj}

If this upper bound is true, then it is also tight. To see that, consider the cycle $C_n$, where $n$ is a multiple of $5$ but not $3$.
By Theorem \ref{cccartcycle2d+1}, we have $d(C_n\cart C_n)=5=2d(C_n)+1$.

\subsection*{Acknowledgment}
\small First author's research was supported by Post Doctoral Fellowship at Indian Institute of Technology, Palakkad.

\renewcommand{\baselinestretch}{0.1}

\bibliographystyle{ams}
\bibliography{bibtex}  

\ifx\undefined\bysame
\newcommand{\bysame}{\leavevmode\hbox to3em{\hrulefill}\,}
\fi
\begin{thebibliography}{10}

\bibitem{akb}
S.~Akbari, M.~Motiei, S.~Mozaffari, and S.~Yazdanbod, {\em Cubic graphs with
  total domatic number at least two}, Discuss. Math. Graph Theory {\bf 38}
  (2018), 75--82.

\bibitem{ara}
H.~Aram, S.~M. Sheikholeslami, and L.~Volkmann, {\em On the total domatic
  number of regular graphs}, Trans. Comb. {\bf 1} (2012), 45--51.

\bibitem{aza}
J.~Azarija, M.~A. Henning, and S.~Klav\v{z}ar, {\em ({T}otal) domination in
  prisms}, Electron. J. Combin. {\bf 24} (2017), 1--19.

\bibitem{ber}
R.~Bertolo, P.~R.~J. \"{O}sterg\aa rd, and W.~D. Weakley, {\em An updated table
  of binary/ternary mixed covering codes}, J. Combin. Des. {\bf 12} (2004),
  157--176.

\bibitem{bou}
I.~Bouchemakh and S.~Ouatiki, {\em On the domatic and the total domatic numbers
  of the 2-section graph of the order-interval hypergraph of a finite poset},
  Discrete Math. {\bf 309} (2009), 3674--3679.

\bibitem{bre}
B.~Bre\v{s}ar, T.~R. Hartinger, T.~Kos, and M.~Milani\v{c}, {\em On total
  domination in the {C}artesian product of graphs}, Discuss. Math. Graph Theory
  {\bf 38} (2018), 963--976.

\bibitem{bu}
Y.~Bu, D.~Chen, A.~Raspaud, and W.~Wang, {\em Injective coloring of planar
  graphs}, Discrete Appl. Math. {\bf 157} (2009), 663--672.

\bibitem{chen}
B.~Chen, J.~H. Kim, M.~Tait, and J.~Verstraete, {\em On coupon colorings of
  graphs}, Discrete Appl. Math. {\bf 193} (2015), 94--101.

\bibitem{cok1}
E.~J. Cockayne, R.~M. Dawes, and S.~T. Hedetniemi, {\em Total domination in
  graphs}, Networks {\bf 10} (1980), 211--219.

\bibitem{cok}
E.~J. Cockayne and S.~T. Hedetniemi, {\em Towards a theory of domination in
  graphs}, Networks {\bf 7} (1977), 247--261.

\bibitem{cran}
D.~W. Cranston, S.-J. Kim, and G.~Yu, {\em Injective colorings of sparse
  graphs}, Discrete Math. {\bf 310} (2010), 2965--2973.

\bibitem{god}
W.~Goddard and M.~A. Henning, {\em Thoroughly dispersed colorings}, J. Graph
  Theory {\bf 88} (2018), 174--191.

\bibitem{grav2}
S.~Gravier, {\em Total domination number of grid graphs}, Discrete Appl. Math.
  {\bf 121} (2002), 119--128.

\bibitem{grav1}
S.~Gravier, M.~Mollard, and C.~Payan, {\em Variations on tilings in the
  {M}anhattan metric}, Geom. Dedicata {\bf 76} (1999), 265--273.

\bibitem{hah}
G.~Hahn, J.~Kratochv\'{\i}l, J.~\v{S}ir\'{a}\v{n}, and D.~Sotteau, {\em On the
  injective chromatic number of graphs}, Discrete Math. {\bf 256} (2002),
  179--192.

\bibitem{har}
F.~Harary and M.~Livingston, {\em Independent domination in hypercubes}, Appl.
  Math. Lett. {\bf 6} (1993), 27--28.

\bibitem{hav}
I.~Havel, {\em Domination in {$n$}-cubes with diagonals}, Math. Slovaca {\bf
  48} (1998), 105--115.

\bibitem{heg}
P.~Heggernes and J.~A. Telle, {\em Partitioning graphs into generalized
  dominating sets}, Nordic J. Comput. {\bf 5} (1998), 128--142.

\bibitem{hen1}
M.~A. Henning, {\em A survey of selected recent results on total domination in
  graphs}, Discrete Math. {\bf 309} (2009), 32--63.

\bibitem{hen2}
M.~A. Henning and D.~F. Rall, {\em On the total domination number of
  {C}artesian products of graphs}, Graphs Combin. {\bf 21} (2005), 63--69.

\bibitem{ho}
P.~T. Ho, {\em A note on the total domination number}, Util. Math. {\bf 77}
  (2008), 97--100.

\bibitem{joh}
S.~M. Johnson, {\em A new lower bound for coverings by rook domains}, Util.
  Math. {\bf 1} (1972), 121--140.

\bibitem{sand}
S.~Klav\v{z}ar and N.~Seifter, {\em Dominating {C}artesian products of cycles},
  Discrete Appl. Math. {\bf 59} (1995), 129--136.

\bibitem{koi}
M.~Koivisto, P.~Laakkonen, and J.~Lauri, {\em N{P}-completeness results for
  partitioning a graph into total dominating sets}, Theoret. Comput. Sci. {\bf
  818} (2020), 22--31.

\bibitem{luz}
B.~Lu\v{z}ar, R.~\v{S}krekovski, and M.~Tancer, {\em Injective colorings of
  planar graphs with few colors}, Discrete Math. {\bf 309} (2009), 5636--5649.

\bibitem{nag}
Z.~L. Nagy, {\em Coupon-coloring and total domination in {H}amiltonian planar
  triangulations}, Graphs Combin. {\bf 34} (2018), 1385--1394.

\bibitem{oh}
S.~Oh, H.~Yoo, and T.~Yun, {\em Rainbow graphs and switching classes}, SIAM J.
  Discrete Math. {\bf 27} (2013), 1106--1111.

\bibitem{van}
G.~J.~M. Wee, {\em Improved sphere bounds on the covering radius of codes},
  IEEE Trans. Inform. Theory {\bf 34} (1988), 237--245.

\bibitem{wol}
A.~J. Woldar, {\em Rainbow graphs}, Codes and designs ({C}olumbus, {OH}, 2000),
  Ohio State Univ. Math. Res. Inst. Publ., vol.~10, de Gruyter, Berlin, 2002,
  pp.~313--322.

\bibitem{zel1}
B.~Zelinka, {\em Domatically ciritical graphs}, Czechoslovak Math. J. {\bf 30}
  (1980), 486--489.

\bibitem{zel4}
B.~Zelinka, {\em Domatic numbers of cube graphs}, Math. Slovaca {\bf 32}
  (1982), 117--119.

\bibitem{zel3}
B.~Zelinka, {\em Regular totally domatically full graphs}, Discrete Math. {\bf
  86} (1990), 71--79.

\bibitem{zel2}
B.~Zelinka, {\em Total domatic number and degrees of vertices of a graph},
  Math. Slovaca {\bf 39} (1989), 7--11.

\end{thebibliography}
\end{titlepage}
\end{document}